\documentclass{amsart}
\usepackage{mathptmx, bm, amscd, amssymb, enumerate, colonequals, mathdots, xcolor}
\definecolor{chianti}{rgb}{0.6,0,0}
\definecolor{meretale}{rgb}{0,0,.6}
\definecolor{leaf}{rgb}{0,.35,0}
\usepackage[colorlinks=true, pagebackref, hyperindex, citecolor=meretale, urlcolor=leaf, linkcolor=chianti]{hyperref}

\newtheorem{theorem}{Theorem}[section]
\newtheorem{lemma}[theorem]{Lemma}
\newtheorem{corollary}[theorem]{Corollary}
\newtheorem{proposition}[theorem]{Proposition}

\theoremstyle{definition}

\newtheorem{question}[theorem]{Question}
\numberwithin{equation}{theorem}
\newtheorem{examplex}[theorem]{Example}
\newenvironment{example}{\pushQED{\qed}\examplex}{\popQED\endexamplex}

\def\rank{\operatorname{rank}}

\def\sat{\operatorname{sat}}

\def\Hom{\operatorname{Hom}}

\def\fraka{\mathfrak{a}}
\def\frakb{\mathfrak{b}}
\def\frakc{\mathfrak{c}}
\def\frakm{\mathfrak{m}}

\def\frakp{\mathfrak{p}}

\def\CC{\mathbb{C}}
\def\FF{\mathbb{F}}
\def\FFp{\mathbb{F}_{\!p}}

\def\ZZ{\mathbb{Z}}

\def\ge{\geqslant}
\def\le{\leqslant}
\def\phi{\varphi}
\def\bar{\overline}
\def\tilde{\widetilde}

\def\to{\longrightarrow}
\def\mapsto{\longmapsto}

\def\AA{\mathbb{A}}

\def\ZZ{\mathbb{Z}}

\def\bsr{{\boldsymbol{r}}}

\newcommand{\into}{\lhook\joinrel\longrightarrow}

\begin{document}
\title{On an ideal membership problem}

\author{Melvin Hochster}
\address{Department of Mathematics, East Hall, 530 Church St., Ann Arbor, MI 48109, USA}
\email{hochster@umich.edu}

\author{Jack Jeffries}
\address{Department of Mathematics, University of Nebraska-Lincoln, 203 Avery Hall, Lincoln, NE~68588, USA}
\email{jack.jeffries@unl.edu}

\author{Nawaj KC}
\address{Department of Mathematics, University of Utah, 155 South 1400 East, Salt Lake City, UT~84112, USA}
\email{nawaj.kc@utah.edu}

\author{Anurag K. Singh}
\address{Department of Mathematics, University of Utah, 155 South 1400 East, Salt Lake City, UT~84112, USA}
\email{singh@math.utah.edu}

\thanks{M.H. was supported by NSF grant DMS~2200501, J.J. by NSF CAREER award DMS~2044833, NSF grant DMS~2601940, and Simons travel grant TSM~00024960, and A.K.S. by NSF grant DMS~2349623.}

\begin{abstract}
Suppose $f_1, \dots, f_{n+1}$ are elements of a regular ring $R$, where $n \colonequals\dim R$. The first author proved that $f_1^n \cdots f_{n+1}^n \in (f_1^{n+1}, \dots, f_{n+1}^{n+1})R$; this uses the Brian\c con-Skoda theorem --- no simpler proof is known, as far as we are aware. We show here that this result is optimal in many respects. On the other hand, when the $f_i$ are homogeneous general polynomials of a fixed degree in $R\colonequals \mathbb{F}[x_1,\dots,x_n]$, for $\mathbb{F}$ a field of characteristic zero, we prove that $f_1 \cdots f_{n+1} \in (f_1^2, \dots, f_{n+1}^2)R$.
\end{abstract}
\maketitle

\section{Introduction}

Let $f_1, \dots, f_{n+1}$ be elements of a noetherian ring $R$, where $n \colonequals\dim R$. Since the local cohomology module $H^{n+1}_{(f_1,\, \dots,\, f_{n+1})}(R)$ vanishes, there exists an integer $k$ such that
\begin{equation}
\label{equation:main:ideal:membership}
\tag{$\dagger$}
f_1^k \cdots f_{n+1}^k\ \in\ (f_1^{k+1},\ \dots,\ f_{n+1}^{k+1})R.
\end{equation}
Note that if this equation holds for some fixed value of $k$, it also holds for any larger value. Without additional hypotheses on $R$, the smallest integer $k$ such that equation~\eqref{equation:main:ideal:membership} holds can be arbitrarily large in a fixed dimension~$n$, as is evident by considering the hypersurface
\[
\frac{\FF[x_1, \dots, x_{n+1}]}{\big(x_1^k \cdots x_{n+1}^k\ -\ \sum_{i=1}^{n+1} x_i^{k+1}\big)}
\]
and taking $f_i \colonequals x_i$. However, when $R$ is a \emph{regular} ring of dimension $n$, the first author proved, as an application of the Brian\c con-Skoda theorem~\cite{SB, Lipman:Teissier, Lipman:Sathaye, HH:JAMS}, that
\[
f_1^n \cdots f_{n+1}^n\ \in\ (f_1^{n+1},\ \dots,\ f_{n+1}^{n+1})R,
\]
i.e., that $k=n$ suffices in~\eqref{equation:main:ideal:membership}; see, for example,~\cite[page~29]{Hochster:615}. The present note explores the extent to which this bound is optimal:

For homogeneous \emph{general} polynomials of a fixed degree, Theorem~\ref{theorem:general} implies that the bound is far from optimal, since $k=1$ suffices when $\FF$ has characteristic zero. Some results specific to positive characteristic are recorded in Section~\ref{sec:small:char}, along with examples demonstrating the optimality of these. Section~\ref{section:dim:three} explores the case of polynomial rings of dimension three, while the optimality of the $k=n$ bound coming from the Brian\c con-Skoda theorem---for polynomial rings of positive prime characteristic---is established in Section~\ref{section:syzygies}; this uses the idea of syzygies of low degree.

Local rings in this paper are assumed to be noetherian.

\section{Homogeneous general polynomials}

\begin{theorem}
\label{theorem:general}
Fix positive integers $n$ and $d$. Let $R \colonequals \FF[x_1,\dots,x_n]$ be a polynomial ring over an infinite field $\FF$, and let $f_1, \dots, f_{n+1}$ be homogeneous general polynomials of degree~$d$. Suppose $\FF$ has characteristic zero or characteristic exceeding $2dn-n$. Then
\[
f_1 \cdots f_{n+1}\ \in\ (f_1^2,\ \dots,\ f_{n+1}^2)R.
\]
\end{theorem}

\begin{proof}
Let $\AA$ denote affine space of dimension
\[
(n+1)\binom{d+n-1}{n-1},
\]
parameterizing the coefficients of the $n+1$ polynomials $f_j$, each of degree $d$. It suffices to prove that there is a dense open subset of $\AA$ for which
\begin{equation}
\label{equation:degree:hilbert}
{[R/(f_1^2,\ \dots,\ f_{n+1}^2)R]}_{dn+d}\ =\ 0.
\end{equation}
Consider the matrix $M$ whose columns are indexed by the monomials of degree $dn+d$ in the indeterminates $x_i$, and the row entries are the coefficients that occur in $f^2_j\mu$ as~$\mu$ runs through the monomials of degree $dn-d$ in the $x_i$, and $1\le j\le n+1$. Specifically, the number of rows and columns of $M$ is
\[
r \colonequals (n+1)\binom{dn-d+n-1}{n-1} \quad\text{ and }\quad c \colonequals \binom{dn+d+n-1}{n-1}
\]
respectively. There is a Zariski \emph{open} set of polynomials~$f_j$, defined by the complement of the vanishing of the ideal of minors $I_c(M)$, for which \eqref{equation:degree:hilbert} holds, where we take $I_c(M)=0$ if $r<c$. It remains to show that this open set is \emph{nonempty}, hence dense.

If the field $\FF$ has characteristic zero, or characteristic greater than $2dn-n$, then this follows from \cite{Stanley} and~\cite[Theorem~3.1]{Phuong:Tran} as we justify next. Set
\[
f_1 \colonequals x_1^d,\ \dots,\ f_n \colonequals x_n^d, \quad f_{n+1} \colonequals (x_1+\dots+x_n)^d
\]
and
\[
A \colonequals R/(f_1^2,\ \dots,\ f_n^2)R,
\]
which is an artinian complete intersection ring. By the above references, the ring $A$ has the strong Lefschetz property, with the map
\begin{equation}
\label{equation:lefschetz}
\CD
{[A]}_{dn-d} @>f_{n+1}^2>> {[A]}_{dn+d}
\endCD
\end{equation}
having maximal rank. Since the Hilbert function of~$A$ is symmetric and unimodal, and the socle of $A$ lives in degree $2dn-n$, one has
\[
\rank_\FF {[A]}_{dn-d}\ \ge\ \rank_\FF {[A]}_{dn+d}
\]
so the map~\eqref{equation:lefschetz} must be surjective. But then 
\[
{[A/f_{n+1}^2A]}_{dn+d}\ =\ 0,
\]
so the Zariski open set $\AA\smallsetminus V(I_c(M))$ is indeed nonempty.
\end{proof}

In positive characteristic, the strong Lefschetz property can fail for monomial complete intersections, so the previous argument does not work in general. However, one has the following weaker result in arbitrary characteristic:

\begin{theorem}
\label{theorem:general:p}
Fix positive integers $n$ and $d$. Let $R \colonequals \FF[x_1,\dots,x_n]$ be a polynomial ring over an infinite field $\FF$, and let $f_1, \dots, f_{n+1}$ be homogeneous general polynomials of degree~$d$. Then, for $k\ge n/2$, one has
\[
f_1^k \cdots f_{n+1}^k\ \in\ (f_1^{k+1},\ \dots,\ f_{n+1}^{k+1})R.
\]
\end{theorem}

\begin{proof}
We may assume that $n\ge 2$. As in the proof of Theorem~\ref{theorem:general}, it suffices to show that, for any given degree $d$ and $k\ge n/2$, there exist $f_1,\dots,f_{n+1}$ of degree $d$ such that 
\begin{equation}
\label{eq:lefschetz:p}
{[R/(f_1^{k+1},\ \dots,\ f_{n+1}^{k+1})R]}_{dk(n+1)}\ =\ 0.
\end{equation}
Write $d=qn+r$ with $0<r \le n$, and set
\[
f_1 \colonequals x_1^d,\ \dots, \ f_n \colonequals x_n^d, \quad
f_{n+1} \colonequals (x_1 \cdots x_r)^{q+1} (x_{r+1} \cdots x_n)^q.
\]
We claim that \eqref{eq:lefschetz:p} holds for this choice of elements $f_i$. Indeed, the highest degree of a nonzero monomial in $R/(f_1^{k+1},\ \dots,\ f_{n+1}^{k+1})R$ is
\[
D_{d,n,k} \colonequals (k+1)(q+1)-1 + (n-1)\big( (k+1)d-1 \big),
\]
as attained, for example, by $x_1^{(k+1)(q+1)-1}(x_2 \cdots x_n)^{(k+1)d-1}$, so it suffices to show that
\[
 D_{d,n,k} < dk(n+1)
\]
whenever $k\ge n/2$. This is the content of the following lemma:
\end{proof}

\begin{lemma}
Let $d,k,n,q$ be integers where $n\ge 2$, $k\ge n/2$, and $0\le q<d/n$. Then 
\[
(k+1)(q+1)-1 + (n-1)\big( (k+1)d-1 \big)\ <\ dk(n+1).
\]
\end{lemma}

\begin{proof}
Consider the function
\[
F(n,d) \colonequals \frac{d(n^2-n+1)-(n^2-n)}{d(2n-1)-n}.
\]
Since
\[
\frac{\partial F}{\partial d}\ =\ \frac{n^2(n-2)}{(d(2n-1)-n)^2}
\]
is nonnegative for any fixed $n\ge 2$, the function $F$ is increasing in $d$. Hence
\[
F(n,d)\ \le\ \lim_{d\rightarrow\infty} F(n,d)\ =\ \frac{n^2-n+1}{2n-1}\ \le\ \frac{n}{2}.
\]
The desired inequality simplifies as
\[
k(2d-1)-(n-1)(d-1)\ >\ (k+1)q.
\]
Since $q<d/n$, it suffices to show that
\[
k(2d-1)-(n-1)(d-1)\ \ge\ (k+1)d/n,
\]
and this is equivalent to $k \ge F(n,d)$. 
\end{proof}

\section{Improved bounds in small characteristic}
\label{sec:small:char}

We record an improvement on the exponent $k$ for which \eqref{equation:main:ideal:membership} holds, in the case of regular rings of positive prime characteristic not exceeding the dimension of the ring:

\begin{theorem}
\label{theorem:small:characteristic}
Let $R$ be a regular ring of positive prime characteristic $p$; set $n \colonequals \dim R $. Then, for all $f_1,\dots, f_{n+1}\in R$, one has
\[
(f_1 \cdots f_n)^{p-1} f_{n+1}\ \in\ (f_1^p,\ \dots,\ f_{n+1}^p)R.
\] 
In particular, the ideal membership~\eqref{equation:main:ideal:membership} holds for $k\ge p-1$.
\end{theorem}

An immediate consequence:

\begin{corollary}
\label{corollary:characteristic:two}
Let $R \colonequals \FF[x_1,\dots,x_n]$ be a polynomial ring over a field $\FF$ of characteristic two. Then, for all $f_1,\dots,f_{n+1}\in R$, one has
\[
f_1 \cdots f_{n+1}\ \in\ (f_1^2,\ \dots,\ f_{n+1}^2)R.
\]
\end{corollary}

When $R$ is local, the following lemma reduces the proof of Theorem~\ref{theorem:small:characteristic} to the case where the first $n$ elements $f_1,\dots,f_n$ form a system of parameters:

\begin{lemma}
\label{lemma:parameters}
Let $(R,\frakm)$ be a local ring of dimension $n$, and fix $F_0,\dots,F_m \in R[x_1,\dots,x_t]$. Consider the statement
\begin{equation}
\label{eqn:containment:specialization}
F_0(\bsr)\ \in\ \big(F_1(\bsr),\ \dots,\ F_m(\bsr)\big)R, \quad \text{ where } \quad \bsr \colonequals r_1,\dots,r_t\in \frakm.
\end{equation}

If $t\le n$, then \eqref{eqn:containment:specialization} holds for all $\bsr$ if and only if it holds whenever $\bsr$ is part of a system of parameters for $R$.

If $t>n$, then~\eqref{eqn:containment:specialization} holds for all $\bsr$ if and only if it holds whenever the first $n$ elements, namely $r_1,\dots,r_n$, form a system of parameters for $R$.
\end{lemma}

\begin{proof}
Suppose first that $t\le n$. Given elements $r_1,\dots,r_n\in R$ and an integer $N\ge 1$, using prime avoidance, one may construct a system of parameters $\tilde{r}_1,\dots,\tilde{r}_n$ such that ${\tilde{r}_i-r_i \in \frakm^N}$ for $1\le i\le n$. The statement
\[
F_0(\tilde{\bsr})\ \in\ \big(F_1(\tilde{\bsr}),\ \dots,\ F_m(\tilde{\bsr})\big)R,
\]
where $\tilde{\bsr} \colonequals \tilde{r}_1,\dots,\tilde{r}_t$, implies that
\[
F_0(\bsr)\ \in\ \big(F_1(\bsr),\ \dots,\ F_m(\bsr)\big)R + \frakm^N.
\]
As $N$ was arbitrary, the Krull Intersection Theorem yields $F_0(\bsr) \in \big(F_1(\bsr),\dots,F_m(\bsr)\big)R$. The case where $t>n$ is much the same.
\end{proof}

\begin{proof}[Proof of Theorem~\ref{theorem:small:characteristic}]
If the statement is false, take a counterexample with $n$ minimal, and let~$\frakp$ be a prime ideal containing
\[
(f_1^p,\dots,f_{n+1}^p)R:_R(f_1 \cdots f_n)^{p-1} f_{n+1}.
\]
Suppose $\dim R_\frakp$ has dimension $m<n$. Then
\[
(f_1 \cdots f_m)^{p-1} f_{m+1}\ \in\ (f_1^p,\dots, f_{m+1}^p)R_\frakp,
\]
so $(f_1 \cdots f_n)^{p-1} f_{n+1} \in (f_1^p,\dots, f_{n+1}^p)R_\frakp$, which is a contradiction. Thus, we have a counterexample in a regular \emph{local} ring of dimension $n$.

Assuming $R$ is local, by Lemma~\ref{lemma:parameters} we reduce to the case where $f_1,\dots,f_n$ is a system of parameters for $R$. Set $\fraka \colonequals (f_1,\dots,f_n)$ and $f_0 \colonequals f_1 \cdots f_n$. Note that
\[
\fraka:_R (\fraka:_R f_{n+1})\ =\ \fraka + f_{n+1}R
\]
by duality in the artinian complete intersection ring $R/\fraka$. Given $r \in \fraka:_R f_{n+1}R$, one has
\[
f_0^{p-1} f_{n+1}r^p\ =\ (f_0r)^{p-1} f_{n+1} r\ \in\ (f_0r)^{p-1}\fraka\ \subseteq\ \fraka^{[p]},
\]
i.e.,
\[
f_0^{p-1} f_{n+1}\, (\fraka:_R f_{n+1}R)^{[p]}\ \subseteq\ \fraka^{[p]}.
\]
It follows that $f_0^{p-1} f_{n+1}$ belongs to the ideal 
\[
\fraka^{[p]}:_R (\fraka:_R f_{n+1})^{[p]}
\ =\ \big(\fraka :_R (\fraka:_R f_{n+1}) \big)^{[p]}
\ =\ \big(\fraka + f_{n+1}R\big)^{[p]},
\]
where the first equality uses the flatness of the Frobenius endomorphism.
\end{proof}

Let $\FF$ be a field and set $R \colonequals \FF[x, y]$. For linear forms $f,g,h \in \FF[x,y]$, one sees readily that $fgh \in (f^2,\ g^2,\ h^2)R$. Given homogeneous polynomials $f, g, h\in R$ of degree $d \ge 2$, the optimal exponent in~\eqref{equation:main:ideal:membership} depends on the characteristic; when $\FF$ has characteristic two, Corollary~\ref{corollary:characteristic:two} implies that $fgh \in (f^2,\ g^2,\ h^2)R$, while this fails in every other characteristic:

\begin{example}
\label{example:dim:2}
Let $\FF$ be a field of characteristic other than two. We construct homogeneous elements $f,g,h \in R \colonequals\FF[x,y]$ of degree $d\ge 2$, for which $fgh \notin (f^2,\ g^2,\ h^2)R$. Set
\[
f \colonequals x^d-y^d, \quad g \colonequals x^d+y^d, \quad h \colonequals x^{d-1}y.
\]
Suppose, to the contrary, that $fgh \in (f^2,\ g^2,\ h^2)R$. Since $2$ is invertible in $\FF$, we have
\[
x^{d-1}yfg\ \in\ (f^2,\ x^dy^d,\ x^{2d-2}y^2)R.
\]
The elements $f^2$ and $x^{d-1}y$ form a regular sequence in $R$, so the above implies that
\[
fg\ \in\ (f^2,\ xy^{d-1},\ x^{d-1}y)R\ \subseteq\ (f^2,\ xy)R.
\]
Using next that $xy$ and $f$ form a regular sequence, one has
\[
g\ \in\ (f,\ xy)R,
\]
and hence, using once again that $2$ is invertible,
\[
x^d\ \in\ (f,\ xy)R.
\]
Lastly, $f$ and $x$ form a regular sequence, so
\[
x^{d-1}\ \in\ (f,\ y)R\ =\ (x^d,\ y)R,
\]
which is clearly false.
\end{example}

\begin{example}
\label{example:dim:2:mixed}
The bound from the Brian\c con-Skoda theorem is also optimal for mixed characteristic regular local rings of dimension two:
Let $p$ be a positive prime integer; set $R \colonequals \ZZ[x]_{(p,\,x)}$, which is a regular local ring of mixed characteristic $p$. As in Example~\ref{example:dim:2}, we shall see that $fgh \notin(f^2,\ g^2,\ h^2)R$, where
\[
f \colonequals x^2-p^2,\quad g \colonequals x^2+p^2,\quad h \colonequals 2px.
\]
Suppose, to the contrary, that $fgh \in (f^2,\ g^2,\ h^2)R$. Since $g^2-f^2=h^2$, one then has
\[
fgh\ \in\ (f^2,\ h^2)R.
\]
Since $f$ and $h$ form a regular sequence in $R$, the above implies
\[
g\ \in\ (f,\ h)R,
\]
equivalently
\[
2x^2\ \in\ (f,\ 2px)R.
\]
But $f$ and $2x$ also form a regular sequence, so
\[
x\ \in\ (f,\ p)R\ =\ (x^2,\ p)R,
\]
a contradiction.
\end{example}

\begin{example}
We next note that Corollary~\ref{corollary:characteristic:two} is optimal: for $\FF$ a field of characteristic other than two, there exist $f_1,\dots,f_{n+1}$ in $R \colonequals \FF[x_1,\dots,x_n]$ of degree $d\ge 2$, with
\[
f_1 \cdots f_{n+1}\ \notin\ (f_1^2,\ \dots,\ f_{n+1}^2)R,
\]
as we record here. Example~\ref{example:dim:2} handles the case $n=2$; suppose $n\ge 3$, take
\[
f_{n-1} \colonequals x_{n-1}^d-x_n^d, \quad f_n \colonequals x_{n-1}^d+x_n^d, \quad f_{n+1} \colonequals x_{n-1}^{d-1}x_n,
\]
and $f_i \colonequals x_i^d$ for $1\le i\le n-2$.
\end{example}

In Example~\ref{example:optimal} we shall see that for $n \ge 2$, and all suitably large integers $d$, there exist polynomials $f_1,\dots,f_{n+1}$ of degree $d$ in $\FFp[x_1,\dots,x_n]$ for which 
\[
(f_1 \cdots f_{n+1})^{n-1}\ \notin\ (f_1^n,\ \dots,\ f_{n+1}^n)R.
\]
These examples being in positive prime characteristic, one\footnote{This question appeared on MathOverflow in May 2026. We had been working on the questions in this manuscript since May 2025, and privately communicated some preliminary results and questions in March~2026; we discovered the MathOverflow post recently, which has compelled us to post this preliminary version.} may ask:

\begin{question}
\label{question:squares:cubes}
Fix $d\ge 2$. For homogeneous polynomials $f_1,f_2,f_3,f_4$ in $R \colonequals \CC[x_1,x_2,x_3]$ of degree $d$, does one always have 
\[
(f_1f_2f_3f_4)^2\ \in\ (f_1^3,\ f_2^3,\ f_3^3,\ f_4^3)R\,?
\]
\end{question}

The answer to Question~\ref{question:squares:cubes} is affirmative if $d\le 3$, as we record in the next section.

\section{Polynomial rings of dimension three}
\label{section:dim:three}

\begin{proposition}
\label{proposition:dim:3}
Let $\FF$ be a field; fix an integer $d \le 3$. If $f_1,f_2,f_3,f_4$ are homogeneous polynomials of degree $d$ in $R \colonequals \FF[x_1,x_2,x_3]$, then
\[
(f_1f_2f_3f_4)^2\ \in\ (f_1^3,\ f_2^3,\ f_3^3,\ f_4^3)R.
\]
\end{proposition}

\begin{proof}
We may assume that the field $\FF$ is algebraically closed. By Theorem~\ref{theorem:small:characteristic}, if $\FF$ has characteristic two or three, the stated containment holds for any four elements of $R$, so we may assume henceforth that $\FF$ has characteristic other than two or three. Suppose first that the $f_i$ generate an ideal of height three. Let $g_1,g_2,g_3,g_4$ be general linear combinations of~$f_i^3$, in which case $g_1,g_2,g_3$ form a regular sequence in $R$. The artinian ring
\[
R/(g_1, g_2, g_3)R
\]
has socle degree $9d-3$, while $(f_1f_2f_3f_4)^2$ has degree $8d$. Since $8d>9d-3$ if~$d<3$, we may assume that $d=3$. If the ideal membership in the proposition does not hold, then 
\[
R/(g_1, g_2, g_3, g_4)R
\]
also has socle degree $24$, so $g_4$ is a linear combination of $g_1$, $g_2$, $g_3$. But then the $f_i^3$ are linearly dependent; after replacing~$f_i$ by scalar multiples and renumbering, the possibilities are $f_1^3=0$, or $f_1^3+f_2^3=0$, or $f_1^3+f_2^3+f_3^3=0$, or $f_1^3+f_2^3+f_3^3+f_4^3=0$. The first two are easily handled; in the third case, it follows from \cite[5.6]{HH93} that $f_3\in (f_1,f_2)R$, so again the assertion holds. In the fourth case, the assignment $z_i\mapsto f_i$ defines a finite map
\[
\FF[z_1,z_2,z_3,z_4]/(z_1^3 + z_2^3 + z_3^3 + z_4^3)\ \to\ R.
\]
By our assumption on the characteristic, the hypersurface to the left is a normal domain with divisor class group $\ZZ^6$, while $R$ is a UFD, so such a finite map does not exist.

Suppose next that $\frakb \colonequals (f_1^3, f_2^3, f_3^3, f_4^3)R$ has height two. Take $g_1,g_2,g_3$ to be general linear combinations of the generators of $\frakb$, so that the minimal primes of $\fraka \colonequals (g_1,g_2,g_3)R$ and~$\frakb$ coincide, and $g_1,g_2$ gives a minimal reduction of $\frakb$ at each of these height two minimal primes $\frakp$. Next, note that $(f_1f_2f_3f_4)^2$ belongs to the integral closure of $\frakb^2R_\frakp$, so the Brian\c con-Skoda theorem gives
\[
(f_1f_2f_3f_4)^2\ \in\ \bar{(g_1,g_2)^2R_\frakp}\ \subseteq\ (g_1,g_2)R_\frakp\ \subseteq\ \fraka R_\frakp.
\]
It follows that $(f_1f_2f_3f_4)^2$ belongs to the saturation $\fraka^{\sat}$ of $\fraka$ with respect to the homogeneous maximal ideal $\frakm$ of $R$. It suffices to show that
\[
\fraka^{\sat}/\fraka\ =\ H^0_\frakm(R/\fraka)
\]
vanishes in degree $8d$. Set $\bar{R} \colonequals R/(g_1,g_2)R$, which is a complete intersection ring of dimension one. The exact sequence
\[
\CD
0 @>>> g_3\bar{R} @>>> \bar{R} @>>> R/\fraka @>>>0
\endCD
\]
yields
\[
H^0_\frakm(R/\fraka)\ \into\ H^1_\frakm(g_3\bar{R}),
\]
so it suffices to show that ${[H^1_\frakm(g_3\bar{R})]}_{8d}=0$. Next consider the exact sequence
\[
\CD
0 @>>> \frakc(-3d) @>>> \bar{R}(-3d) @>g_3>> g_3\bar{R} @>>> 0,
\endCD
\]
where $\frakc \colonequals (0:_{\bar{R}} g_3)$, and the induced sequence
\[
\CD
H^1_\frakm(\frakc)(-3d) @>>> H^1_\frakm(\bar{R})(-3d) @>g_3>> H^1_\frakm(g_3\bar{R}) @>>> 0.
\endCD
\]
Note that
\[
{[H^1_\frakm(\bar{R})(-3d)]}_{8d}\ =\ {[H^1_\frakm(\bar{R})]}_{5d}
\]
vanishes in view of $a(\bar{R})=6d-3$ if $d<3$. If $d=3$, one needs to further verify that
\[
\CD
{[H^1_\frakm(\frakc)]}_{15} @>>> {[H^1_\frakm(\bar{R})]}_{15}
\endCD
\]
is surjective. But the dual to this is the injective map
\[
\CD
{[\Hom_{\bar{R}}(\bar{R},\ \bar{R})]}_0 @>>> {[\Hom_{\bar{R}}(\frakc,\ \bar{R})]}_0,
\endCD
\]
where the identity on $\bar{R}$ maps to the inclusion $\frakc \into \bar{R}$.

Lastly, if $(f_1,f_2,f_3,f_4)R$ has height one, then the polynomials $f_i$ have a common factor. Factoring that out, we reduce to a case that has been previously handled.
\end{proof}

\section{Syzygies of low degree, and optimality}
\label{section:syzygies}

Set $R \colonequals \FF[x_1,\dots,x_n]$, where $\FF$ is a field, and $n\ge 2$. Suppose $f_1,\dots,f_{n+1}$ are homogeneous polynomials of degree $d$ in $R$, such that
\begin{enumerate}
\item $(f_1 \cdots f_{n+1})^{n-1} \notin (f_1^n,\dots,f_{n+1}^n)R$;
\item $f_1,\dots,f_n$ form a regular sequence in $R$;
\item $f_{n+1}^n \notin (f_1^n,\dots,f_n^n)$.
\end{enumerate}
Extending an argument used in Proposition~\ref{proposition:dim:3}, we first prove that $d>n$. The artinian complete intersection ring
\[
A \colonequals R/(f_1^n,\ \dots,\ f_n^n)R
\]
has socle degree $dn^2-n$, and (3) additionally forces ${[A/f_{n+1}^nA]}_{dn^2-n}=0$. But then (1) implies that
\[
d(n+1)(n-1)\ <\ dn^2-n,
\]
equivalently $d>n$.

This leads to the question: if $d>n$, are there polynomials satisfying (1), (2), (3) above? Such polynomials, when they exist, necessarily have a syzygy of low degree:

\begin{proposition}\label{proposition:syzygy}
Suppose $f_1,\dots,f_{n+1}$ are homogeneous polynomials of degree $d$ in the polynomial ring $\FF[x_1,\dots,x_n]$, satisfying conditions $(1)$, $(2)$, $(3)$. Then $f_1^n,\dots,f_{n+1}^n$ admit a nonzero syzygy of degree $d-n$. 
\end{proposition}

\begin{proof}
As we have seen, $d \ge n+1$. Consider the artinian complete intersection ring
\[
A \colonequals R/(f_1^n,\ \dots,\ f_n^n)R.
\]
The hypotheses imply that the map
\[
\CD
A_{dn^2-d-dn} @>f_{n+1}^n>> A_{dn^2-d}
\endCD
\]
is \emph{not} surjective. The ring $A$ is Gorenstein, hence has a perfect pairing; since $A$ has socle degree $dn^2-n$, the above implies that 
\[
\CD
A_{d-n} \ \cong\ A_{dn^2-d}^* @>f_{n+1}^n>> A_{dn^2-d-dn}^*\ \cong\ A_{d-n+dn}
\endCD
\]
is not injective. Lifting a nonzero element of the kernel to a homogeneous form $h\in R_{d-n}$, we have $hf_{n+1}^n \in (f_1^n,\dots,f_n^n)R$.
\end{proof}

In Example~\ref{example:dim:2}, we have seen the existence of such syzygies of low degree: in characteristic other than two, if $d\ge 3$, then the polynomials $f \colonequals x^d-y^d$, $g \colonequals x^d+y^d$, $h \colonequals x^{d-1}y$ from that example satisfy the conditions (1), (2), (3), and, indeed, $f^2$, $g^2$, $h^2$ have a syzygy of degree $d-2$, namely
\[
x^{d-2}f^2\ -\ x^{d-2}g^2\ +\ 4y^{d-2}h^2\ =\ 0.
\]

We record elementary positive characteristic examples demonstrating the optimality of the bound coming from the Brian\c con-Skoda theorem in each dimension $n\ge 2$.

\begin{example}
\label{example:optimal}
Fix $n$; for all suitably large integers $d$, there exists a prime integer $p$ with
\[
dn\ <\ p\ \le dn+d-n.
\]
Fix such integers $d$ and $p$. Let $R \colonequals \FFp[x_1,\dots,x_n]$, and set $f_i \colonequals x_i^d$ for $1\le i\le n$, and
\[
f_{n+1} \colonequals \Big(\sum_{i=1}^n x_i\Big)^d.
\]
We claim that
\[
(f_1 \cdots f_{n+1})^{n-1}\ \notin\ (f_1^n,\ \dots,\ f_{n+1}^n)R.
\]
Assume the contrary, i.e., that
\[
(x_1 \cdots x_n)^{dn-d}({\textstyle\sum} x_i)^{dn-d}\ \in\ \big(x_1^{dn},\ \dots,\ x_n^{dn},\ ({\textstyle\sum} x_i)^{dn}\big)R.
\]
Multiplying by $(\sum x_i)^{p-dn}$ then gives
\[
(x_1 \cdots x_n)^{dn-d}({\textstyle\sum} x_i)^{p-d}\ \in\ \big(x_1^{dn},\ \dots,\ x_n^{dn},\ ({\textstyle\sum} x_i)^p\big)R
\ =\ \big(x_1^{dn},\ \dots,\ x_n^{dn}\big)R.
\]
Since $x_1,\dots,x_n$ form a regular sequence in $R$, the above implies that
\[
({\textstyle\sum} x_i)^{p-d}\ \in\ \big(x_1^d,\ \dots,\ x_n^d\big)R.
\]
The multinomial coefficients in the expansion of $({\textstyle\sum} x_i)^{p-d}$ are nonzero in~$\FFp$, so the pigeonhole principle implies that $p-d\ >\ n(d-1)$. This contradicts the assumption on the prime integer $p$.
\end{example}

\subsection*{AI disclosure}
Generative AI was used to search for examples and to proofread the paper; some of the results were assisted by the use of generative AI:

Theorem~\ref{theorem:small:characteristic} and the subsequent corollary were given by generative AI when instructed to proofread an earlier draft in which the $n=2$ case of Corollary~\ref{corollary:characteristic:two} was proven; the proof presented here was reworked by the authors. One specific case of Example~\ref{example:optimal} with $n=3$ was produced by generative AI; the agent was not able to produce an example in dimension three directly, but successfully returned an example after instructed to find one satisfying the conclusion of Proposition~\ref{proposition:syzygy}. Generative AI also produced some parts of the proof of Proposition~\ref{proposition:dim:3} when instructed to find the above examples.

Other results did not have any contribution from generative AI, nor was generative AI used in the writing of the paper.


\end{document}